\documentclass[12pt]{amsart}
\usepackage{amssymb}
\usepackage{graphicx}
\usepackage{enumitem}
\usepackage{slashed}
\usepackage{diagbox}  
\usepackage{makecell}
\usepackage{hyperref}
\usepackage{color}
\usepackage[usenames,dvipsnames,svgnames,table]{xcolor}
\usepackage{soul}

\newcommand{\amscite}[3]{\cite[#3]{#1}}

 \def\ds{\displaystyle}

\def\ds{\displaystyle}
 
\def\1{1\!\mathrm l}

\newtheorem{To}{Theorem}[section]

 \newtheorem{lem}[To]{Lemma}
 \theoremstyle{definition}
 \newtheorem{de}[To]{Definition}

\theoremstyle{remark}
\newtheorem{rem}[To]{Remark}
\begin{document}

\title{Numerically stable conditions on rational and essential singularities}

\author{Amerah Alameer}
\maketitle

\begin{abstract}

 This paper demonstrates some connections between the coefficients  of a
Taylor series $f(z)=\ds\sum_{n=0}^\infty a_n z^n$ and singularities of the function. There are many known
results of this type, for example, counting the number of poles on the circle of convergence,
and doing convergence or overconvergence for $f$ on any arc of holomorphy. A new approach proposed here is that these kinds of results are extended by relaxing
the classical conditions for singularities and convergence theorems. This is done by allowing the
coefficients to be sufficiently  small instead of being zero.
\end{abstract}

\section{Introduction}
It is a classical question to analyse the boundary behaviour of an
analytic function in terms of its Taylor expansion.
 For example, for a function $f$ which is analytic on $\Bbb{E}=\{z\in\Bbb{C} : |z|<1\}$, either $f$ can be extended beyond the circle of convergence $\partial\Bbb{E}$, or $\partial\Bbb{E}$ is the natural boundary (i.e., every $z\in\partial\Bbb{E}$ is a singular point of $f$).

Many of the results that have been obtained
so far (see discussion below) require to vanish some of the Taylor coefficients. Although these results are interesting in
the theoretical sense, they are not really applicable for practical
problems because, in applied science, values may not be
strictly zero. For example, in numerical methods there is no test
that can confirm that such a coefficient is exactly zero: one may
only conclude that it is sufficiently close to zero.

  In this paper, we will 
 address the following two questions for a Taylor series $f(z)=\ds\sum_{n=0}^\infty a_n z^n$ that does not necessarily have gaps.
\begin{enumerate}[label=(\roman*)]
  \item   How many poles of $f$ are on the circle of convergence?
	\item  If $f$ is defined analytically on $\Bbb{E}$, when is  $\partial\Bbb{E}$ the natural boundary for $f$? 
\end{enumerate}

 We now discuss some classical results along with our generalizations, and the first example is the following.
\begin{To}{}\amscite{MR580154}*{\S~5.3, Prob.244}.\label{7}
Let $v_{n}$ be the number of non-zero coefficients among the $n$
coefficients $a_{0},a_{1},...,a_{n-1}$. If there are only poles (and
no other singularities) on the circle of convergence of the power
series: $\ds\sum_{n=0}^{\infty}a_n z^n$, the number
of such poles is not smaller than $\ds\limsup_{n\rightarrow\infty}\frac{n}{v_{n}}$.
\end{To}

Our first object is to replace the number 
$v_n$ in this theorem by the number of coefficients that shall exceed some small values.  This will constitute  Theorem \mbox{\ref{9}}, which has the same conclusion as in Theorem \mbox{\ref{7}} but in a more general setting.

We next consider the classical theorem of Ostrowski convergence for lacunary series. The definition of lacunary series used in the initial theorem is as follows.
 \begin{de} \amscite{remmert2013classical}*{\S~11.1}.\label{23}
An infinite power series $\ds\sum_{v=0}^\infty a_{v}z^{v}$ is called a \emph{lacunary} series if there
exists an increasing sequence $(m_{v})$ of non-negative integers for $v=0,1,2,\dots$  such that $\ds\lim_{v\rightarrow\infty}(m_{v+1}-m_{v})=\infty$,
and 
 \[
a_{j}=0,\quad m_{v}<j<m_{v+1},\, v=0,1,2,\dots.
\]
\end{de} 
 Ostrowski proved the following important property for lacunary series.
\begin{To}[Ostrowski's convergence]\label{11} \amscite{remmert2013classical}*{\S~11.1, Thm 3}.
 If $f(z)=\ds\sum_{n=0}^\infty a_{m_{n}}z^{m_{n}}$ is a lacunary series with
a bounded sequence of coefficients, its sequence of partial sums $s_{m_{n}}$
converges uniformly on every arc of holomorphy $L$ of $f.$ 
\end{To}
 Our treatment of this theorem is to relax the restrictive condition  of consecutive zero coefficients by instead requiring some coefficients of the  Taylor series to be sufficiently small (see Definition \ref{16}). More precisely, in Theorem \mbox{\ref{10}} we relax  the lacunas, whilst at the same time the conclusion of Theorem \ref{11} is preserved.

Another interesting theorem is that of Ostrowski overconvergence for Ostrowski series, which is defined as follows. 
\begin{de} \amscite{remmert2013classical}*{\S~11.2}\label{15}.
A power series $\ds\sum_{v=0}^\infty a_{v}z^{v}$ is called an \emph{Ostrowski} series if 
there exists a $\delta>0$ and two sequences $(m_{v})$ and $(n_{v})$ of non-negative integers for $v=0,1,2,\dots$ such that:
\begin{enumerate}[label=(\roman*)]
\item $0\leq m_{0}<n_{0}\leq m_{1}<n_{1}\leq...\leq m_{v}<n_{v}\leq m_{v+1}...$;
and $\, n_{v}-m_{v}>\delta m_{v},\, v=0,1,2,\dots$; and
\item $a_{j}=0$ if\, $m_{v}<j<n_{v},\, v=0,1,2,\dots$.
\end{enumerate}
\end{de}Let us denote by $B(r,0)$ a disc centred at the origin with radius $r>0$. The classical theorem of Ostrowski's overconvergence
is stated as follows.
\begin{To}[Ostrowski's overconvergence]
\amscite{remmert2013classical}*{\S~11.2, Thm 1}.\label{12} Let
$f(z)=\ds\sum_{v=0}^\infty a_{v}z^{v}$ be an Ostrowski series with radius of convergence
$r>0$, and let $A\subset{\partial B(r,0)}$ denote the set of all
the boundary points of $f$ that are not singular. Then the sequence
of sections $s_{m_{k}}(z)=\ds\sum_{v=0}^{m_{k}}a_{v}z^{v}$ converges
uniformly in a neighbourhood of $B(r,0)\cup A$. 
\end{To}

We again relax the requirement for the gaps of the power series in Theorem \ref{12}. 
In Definition \ref{17}, we replace the consecutive zeros by adequately small coefficients instead, which allows us to prove Theorem \ref{14}---a generalization of Theorem \ref{12}.

Finally, the last classical result that interests us is Hadamard's theorem for Hadamard lacunary series that is defined as follows.
\begin{de}\amscite{remmert2013classical}*{\S~11.2, Def 3}.\label{def series Hadamard} An infinite power
series $\ds\sum_{v=0}^\infty a_{v}z^{v}$ is called a \emph{Hadamard lacunary}  series if  there exists a $\delta>0$ and an increasing sequence $(m_{v})$ of non-negative integers for $v=0,1,2,\dots$ 
such that 
\[
  m_{v+1}-m_{v}>\delta m_{v},\, v=0,1,2,\dots;\,\, a_{j}=0,\,\,\text{if}\,\, m_{v}<j<m_{v+1},\, a_{m_{v}}\neq0.
\]
\end{de}Every Hadamard lacunary series is a lacunary series in the sense of Definition \ref{23}, and
also an Ostrowski series (with $n_{v}=m_{v+1}$). The converse, however, is
not true: for a lacunary series as in Definition \ref{23}, only $\ds\lim_{v\rightarrow\infty}(m_{v+1}-m_{v})=\infty$
is required, whereas for an Ostrowski series gaps need to appear only 
``here
and there''. However, the Hadamard lacunary condition requires
that a gap lies  between any two successive terms that actually appear.

The classical Hadamard's gap theorem is stated as
follows.
\begin{To}[Hadamard's gap theorem]\amscite{remmert2013classical}*{\S~11.2, Thm
3.}\label{Hadamard} Every Hadamard lacunary series $\ds\sum_{v=0}^\infty a_{v}z^{v}$ with
radius of convergence $r>0$ has the disc $B(r,0)$ as a domain of
holomorphy. 
\end{To}
\begin{rem} The condition $a_{m_{v}}\neq0$
in Definition \ref{def series Hadamard} of a Hadamard lacunary series is necessary, otherwise
we would obtain the series with zero coefficients that has trivial convergence
everywhere.
\end{rem} 

Our modification of Theorem \ref{Hadamard} is to once more replace the lacunas of the power series  by suitably  small values (see Definition \ref{18} and Theorem \ref{8}). 
 
 The above classical results were recently addressed by Breuer and Simon in \cite{BreuerSimon11a},  in which drew  a parallel  between the spectral theory of  Jacobi matrices and singular points on the boundary. However, their technique was based on right limits and is only applicable to bounded Taylor coefficients that did not tend to zero. 
On the other hand, our extensions of the classical results are free from both of these limitations.   In view of the similarity between natural boundaries and the absolutely
continuous spectrum of Jacobi matrices stated in \cite{BreuerSimon11a}, 
we  conjecture that there are numerically stable
results in spectral theory that are similar to theorems presented in
this paper.

Furthermore, in \cite{luhpower} it was found that the lacunary property is stable, i.e., if a lacunary series is expanded around another point, it remains a lacunary series. The content of this paper can be viewed as another version of stability for lacunary series  in  the following sense: we consider small alternations of lacunary series that destroy the lacunas but keep the natural boundary conclusion. It is then natural to ask the open question:  are our generalised series (Definitions \ref{16}, \ref{17} and \ref{18}) stable in the sense of \cite{luhpower}?

This topic is a part of a broader picture.
It has already been mentioned that an intriguing connection with spectral theory was discovered in \cite{BreuerSimon11a}. 
On the other hand, Taylor coefficients can be viewed as one of the
fundamental examples of the covariant transform \cite{Kisil97c,Kisil11c}
of analytic functions. Many classical results of harmonic analysis
describe how properties of functions are ``transported'' by the
covariant transform \cite{Kisil12d}. Therefore, the results presented in this
paper show a certain stability in this transportation: a
small variation in the Taylor coefficients preserves certain properties of the function under consideration. 


\section{Poles on the circle of convergence}
In this section, we prove Theorem \ref{9}, which is a generalization of Theorem \ref{7} presented in the Introduction.  These results are  a presentation of the relationship
between coefficients of a power series and singularities of the function
they present. The key to proving Theorem \ref{9} will be  Lemma
\ref{13}.

\begin{lem}\amscite{MR580154}*{\S~5.3, Prob.243}\label{13}
Let $\ds\sum_{n=0}^{\infty}a_{n}z^{n}$ be the expansion into a power
series of a rational function whose denominator ( relative prime to
the numerator) has degree $q$. If $A_n=\max\{|a_{n}|,|a_{n-1}|,...,|a_{n-q+1}|\}$, then the radius of convergence $r$ satisfies

\begin{equation}\label{maximum radius}
	\ds\lim_{n\rightarrow\infty}\sqrt[n]{A_{n}}=\frac{1}{r}.
\end{equation}
\end{lem} To present our modification of Theorem \ref{9}, we consider the number of coefficients
that shall exceed some small values instead of being non-zero in Theorem \ref{7}, and we obtain our new theorem as follows. 

\begin{To}\label{9}
Let $\ds\sum_{n=0}^\infty  a_{n}z^{n}$ be a power series that has only poles (and
no other singularities) on the circle of convergence $|z|=\rho$,
as well as no other singularities inside the circle $|z|=\rho_{1}>\rho$.
Let $\epsilon>0$  such that $\frac{1}{\rho_{1}}<\frac{1}{\rho}-{2}\epsilon$.
If $v_{n}$ is the number of coefficients $|a_{j}|>(\frac{1}{\rho}-\epsilon)^{j}$
among $n$ coefficients $a_{0},a_{1},...,a_{n-1}$, the number of
poles on the circle of convergence is not smaller than $\ds\limsup_{n\rightarrow\infty}\frac{n}{v_{n}}$.
\end{To}

\begin{proof} Let $k$ be the number of poles on the circle of convergence.
According to the hypotheses, we set $\ds\sum_{n=0}^{\infty}$ $ a_{n}z^{n}=\ds\sum_{n=0}^{\infty}\alpha_{n}z^{n}+\sum_{n=0}^{\infty}b_{n}z^{n}$
such that:
\begin{enumerate}[label=(\roman*)]
\item the series $\ds\sum_{n=0}^\infty \alpha_{n}z^{n}$ is the expansion into a power series
of the rational part of $\ds\sum_{n=0}^\infty  a_{n}z^{n}$ that has the radius of
convergence $\rho$; and
\item the series $\ds\sum_{n=0}^\infty  b_{n}z^{n}$ has the radius of convergence
$\rho_{1}$ defined on the hypothesis, i.e. $\ds{\limsup}_{n\rightarrow\infty}\sqrt[n]{|b_{n}|}=\frac{1}{\rho_{1}}<\frac{1}{\rho}-2\epsilon$.
\end{enumerate}
Then, for $\epsilon>0$ given in the hypothesis, there exists $N_{1}\in\Bbb{N}$
such that for all $n\geq N_{1}$, $|b_{n}|<(\frac{1}{\rho_{1}}+\epsilon)^{n}<(\frac{1}{\rho}-\epsilon)^{n}$.
According to $(i)$ and Lemma \ref{13}, we have for $\epsilon$,
 there exists $N_{2}\in\Bbb{N}$ such that for all $n\geq N_{2}$
we have that
\[
\max(|\alpha_{n}|,|\alpha_{n-1}|,...,|\alpha_{n-k+1}|)>\max\left((\frac{1}{\rho}-\epsilon)^{n},...,(\frac{1}{\rho}-\epsilon)^{n-k+1}\right).
\]
Set $N=\max(N_{1},N_{2})$. Then for all $n\geq N$, there is $\bar{n}$, $n\geq\bar{n}\geq n-k+1$ such that $|\alpha_{\bar{n}}|>(\frac{1}{\rho}-\epsilon)^{\bar{n}}>|b_{\bar{n}}|$.  Otherwise, $\max(|\alpha_{n}|,|\alpha_{n-1}|,...,|\alpha_{n-k+1}|)=0$. Then by \eqref{maximum radius}, the radius of convergence of the series $\ds\sum_{n=0}^\infty a_n z^n$ is infinite, which gives a contradiction with the given finite radius.
Thus, $|a_{\bar{n}}|=|\alpha_{\bar{n}}+b_{\bar{n}}|>(\frac{1}{\rho}-\epsilon)^{\bar{n}}$.
Consequently, $\ds\frac{n-N}{k}\leq v_{n}$. Then, $\ds\frac{n}{k}-c\leq v_{n}$,
where $\ds c=\frac{N}{k}$. Thus $ k\geq\ds\limsup_{n\rightarrow\infty}\frac{n}{v_{n}}$.
\end{proof} 

\section{Expansion of Analytic Functions and its convergence on the boundary }

 In this section, we shall generalise a result on the boundary behaviour
of a power series 
(see Theorem \ref{11}), which links the extension
problem for a power series with the convergence of its sequence of partial
sums. 
This is by replacing the consecutive zero coefficients
in a lacunary series (see Definition \ref{23}) with small values at the same places to create
a quasi-lacunary series as follows.

\begin{de}\label{16} An infinite power series $\ds\sum_{n=0}^\infty a_{n}z^{n}$ is called
a \emph{quasi-lacunary} series if there exists an increasing  sequence $(m_{v})$ of  non-negative integers for $v=0,1,2,\dots$
such that $$\lim_{v\rightarrow\infty}(m_{v+1}-m_{v})=\infty,\quad \text{and}\quad |a_{j}|\leq |c_{j}|,\quad m_{v}<j<m_{v+1},\, v=0,1,2,\dots,$$ where $(c_j),\, j=0,1,2,\dots,$ is a $p-$ summable sequences for some $p>1$, i.e. $\ds\sum_{j=0}^\infty |c_{j}|^p<\infty$.

\end{de} 
Let $\Bbb{E}=\{z\in\Bbb{C}:|z|<1\}$. To give the proof of Theorem \ref{10},
we need the following lemma.

\begin{lem}[M.Riesz] \amscite{remmert2013classical}*{\S~11,
Lem1}\label{21}. For every arc of holomorphy $L\subset$ $\partial\Bbb{E}$
of a power series $f(z)=\ds\sum_{v=0}^\infty a_{v}z^{v}$ with radius of convergence
$1$, there exists a compact circular sector $S$ with vertex at $0$
such that $L$ lies in the interior $\mathring{S}$ of $S$ and $f$
has a holomorphic extension $\widehat{f}$ to $S$. Let $z_{1},z_{2}\neq0$
be the corners of $S$, and let $w_{1}$ and $w_{2}$ be the points
of intersection of $\partial\Bbb{E}$ with $[0,z_{1}]$ and $[0,z_{2}]$,
respectively. Then $|w_{1}|=|w_{2}|=1$ and $s=|z_{1}|=|z_{2}|>1$.
\end{lem} To prove the next result, we consider the functions
\begin{equation}
g_{n}(z)=\frac{\widehat{f}(z)-s_{n}(z)}{z^{n+1}}(z-w_{1})(z-w_{2}),\quad\text{where}\quad s_n(z)=\sum_{k=0}^n a_k z^k,\, n\in\Bbb{N}.\label{19}
\end{equation}

So, every function $g_{n}$ is holomorphic in $S$. In the proof of
the next theorem, $S$ and $\widehat{f}$ can be chosen as in Lemma \ref{21} and $g_{n}$ as in \eqref{19}. Denoting by $  (\|{\widehat{f}}-s_{n}\|)|_{L}$ and  $\| g_{n}\|)|_{S}$ be the maximum of the absolute values of the functions $({\widehat{f}}-s_{n})(z)$ and $g_{n}(z)$ for all $z$ in $L$ and $S$, respectively. Since $|z|=1$ for $z\in L$, the inequality 
\begin{equation}
  (\|{\widehat{f}}-s_{n}\|)|_{L}\leq a^{-1} (\| g_{n}\|)|_{S},\quad where\quad a=\min_{z\in L}\lbrace|(z-w_{1})(z-w_{2})|\rbrace>0,\label{22}
\end{equation}
holds for all $n\in\Bbb{N}$.

Now, we can prove our convergence theorem for a quasi-lacunary
series.

\begin{To}[Convergence for a quasi-lacunary
series]\label{10}  Let $f(z)=\ds\sum_{n=0}^\infty a_n z^n$ be a quasi-lacunary
series such that $(a_{m_v})$, where $a_{m_v}$ are defined in  Definition \ref{16}, is a bounded sequence of coefficients. Then the sequence
of partial sums $s_{m_{v}}$ converges uniformly on every arc of holomorphy
$L$ of $f$.
\end{To}

 \begin{proof} It must be shown that  $\ds\lim_{v\rightarrow\infty}(\|{\widehat{f}}-s_{m_{v}}\|)|_{L}=0$.
By \eqref{22}, it suffices to show that the sequence $g_{m_{v}}$
tends  locally uniformly to zero in $\mathring{S}$. Let $t\in(0,1)$. By Vitali's
theorem \amscite{remmert2013classical}*{\S~7.3, Thm.2}, it suffices
to show that $\ds\lim_{v\rightarrow\infty} g_{m_{v}}(z)=0$ for $|z|=t$. Setting $A=\sup|a_{m_{v}}|$,
we have  
\begin{eqnarray*}
( |{\widehat{f}}(z)-s_{m_{v}}(z)|)|_{S} & = &\left.{\left (\left|\sum_{k\in\cup_{n=v}^\infty(m_n,m_{n+1}) }a_{k}z^{k}+\sum_{j=v+1}^{\infty}a_{m_j}z^{m_j}\right|\right)}\,\right|_{S}\\
 & \leq & \sum_{k\in\cup_{n=v}^\infty(m_n,m_{n+1}) }|a_{k}|t^{k}+\sum_{j=v+1}^{\infty}|a_{m_j}|t^{m_j}\,.
\end{eqnarray*}
According to \eqref{19}, we have that
\begin{eqnarray*}
( | g_{m_{v}}(z)|)|_{S} & = & \frac{ ( |\widehat{f}(z)-s_{m_{v}}(z)|)|_{S}}{|z^{m_{v}+1}|}\cdot|z-w_{1}||z-w_{2}|\\
 & \leq & \frac{ ( |\widehat{f}(z)-s_{m_{v}}(z)|)|_{S}}{t^{m_{v}+1}}\cdot(1+t)^{2}\,.
\end{eqnarray*}

Subsequently, 
\begin{equation}
( | g_{m_{v}}(z)|)|_{S}\leq\sum_{k\in\cup_{n=v}^\infty(m_n,m_{n+1}) }|a_{k}|t^{k}\frac{(1+t)^{2}}{t^{m_{v}+1}}+\sum_{j=v+1}^{\infty}|a_{m_j}|t^{m_j}\frac{(1+t)^{2}}{t^{m_{v}+1}}.\label{6}
\end{equation}

The first term of the right hand side in \eqref{6} tends to zero.
In fact, by Holder's inequality for $1<p, q<\infty$ and $\ds\frac{1}{p}+\frac{1}{q}=1$, we have 
 \begin{eqnarray*}\nonumber
\sum_{k\in\cup_{n=v}^\infty(m_n,m_{n+1}) }|a_{k}|t^{k}\frac{(1+t)^{2}}{t^{m_{v}+1}} & \leq & \left(\sum_{k\in\cup_{n=v}^\infty(m_n,m_{n+1}) }|a_{k}|^p\right)^{\frac{1}{p}}\left(\sum_{k\in\cup_{n=v}^\infty(m_n,m_{n+1}) }t^{q k}\right)^{\frac{1}{q}}\frac{(1+t)^{2}}{t^{m_{v}+1}}\\&\leq & \left(\sum_{k\in\cup_{n=v}^\infty(m_n,m_{n+1}) }|a_{k}|^p\right)^{\frac{1}{p}}\left(\frac{t^{q m_v}}{1-t^q}\right)^{\frac{1}{q}}\frac{(1+t)^{2}}{t^{m_{v}+1}}
\\
 & \leq & \left(\sum_{k\in\cup_{n=v}^\infty(m_n,m_{n+1}) }|a_{k}|^p\right)^{\frac{1}{p}}\frac{t^{m_v}}{(1-t^q)^\frac{1}{q}}\cdot\frac{(1+t)^{2}}{t^{m_{v}+1}}\\
 & = & \left(\sum_{k\in\cup_{n=v}^\infty(m_n,m_{n+1})}|a_{k}|^p\right)^{\frac{1}{p}}\frac{(1+t)^{2}}{t(1-t^q)^\frac{1}{q}}\rightarrow0\quad as\quad v\rightarrow\infty.
\end{eqnarray*}
Indeed, by the hypothesis, we have $\ds\sum_{k\in\cup_{n=v}^\infty(m_n,m_{n+1})}|a_{k}|^p<\infty$. Therefore, when $m_{n}\rightarrow\infty$ as $n\rightarrow\infty$,  it follows that
$\ds\sum_{k\in\cup_{n=v}^\infty(m_n,m_{n+1})}|a_{k}|^p\rightarrow0$.

On the other hand, the second term of the right hand side in \eqref{6}
also tends to zero. Indeed, since $\ds\lim_{v\rightarrow\infty}(m_{v+1}-m_{v})=\infty$ and $0<t<1$,
then $\ds\lim_{v\rightarrow\infty} t^{m_{v+1}-m_{v}}=0$,
\begin{align*}
	\sum_{j=v+1}^{\infty}|a_{m_j}|\, t^{m_j}\frac{(1+t)^{2}}{t^{m_{v}+1}} & \leq  A\,\frac{t^{m_{v+1}}}{1-t}\cdot\frac{(1+t)^{2}}{t^{m_{v}+1}}\\
 & =  A\, t^{m_{v+1}-m_{v}}\cdot\frac{(1+t)^{2}}{(1-t)t}\longrightarrow0.
\end{align*}
 Therefore, $\ds\lim_{v\rightarrow \infty}( \| g_{m_{v}}\|)|_{S}=0$.
Then, $\ds\lim_{v\rightarrow\infty}( \|\widehat{f}-s_{m_{v}}\|)|_{L}=0$. Thus, the sequence of partial sums $(s_{m_{v}}(z))$ converges uniformly on every arc
of holomorphy $L$ of $f$. 
\end{proof} 

\section{Expansion of Analytic Functions and Overconvergence theorems}

This section is dedicated to proving  two results (namely, Theorems \ref{14} and \ref{8}), which are further developments
of  Theorems \ref{12} and \ref{Hadamard} stated in the Introduction.
 These classical theorems demonstrate close relationships between
overconvergence and the gaps of the power
series, and once our modifications--- the removal of the restrictive gap conditions--- are applied we are still able to derive the same conclusion. 

To show Theorem \ref{14}, we first need the definition of a quasi-Ostrowski
series. 
\begin{de}\label{17} An infinite power series $\ds\sum_{v=0}^\infty a_{v}z^{v}$ is called a {\emph {quasi-Ostrowski}} series
if there exists a $\delta>0$, a positive decreasing sequence $(c_{r}),\, r=0,1,2,\dots$,
$c_{r}\searrow0$, and two sequences $(m_{k})$ and $(n_{k})$ of non-negative integers for $k=0,1,2,\dots,$ such that: 

\begin{enumerate}[label=(\roman*)]
\item $ n_{k}-m_{k}>\delta m_{k},\,\, k=0,1,2,\dots\,$, and $\,0\leq m_{0}<n_{0}\leq m_{1}<n_{1}\leq...\leq m_{k}<n_{k}\leq m_{k+1}...\,$; and
\item $|a_{j}|<\frac{1}{m_{k}^{2}}c_{j}$,\quad $m_{k}<j<n_{k},\,\, k=0,1,2\dots$.
\end{enumerate}
\end{de}

Recall that $B(r,0)$ denotes the open disc centred at zero and radius $r>0$. Then, we have our generalization of Theorem \ref{12} as follows.

\begin{To}[Overconvergence theorem for a quasi-Ostrowski series]\label{14}
Let $f(z)=\ds\sum_{n=0}^\infty a_{n}z^{n}$ be a quasi-Ostrowski series with radius
of convergence $r>0$, and let $A\subset{\partial B(r,0)}$ denote
the set of all the boundary points of $f$ that are not singular.
Then the sequence of sections $s_{m_{k}}(z)=\ds\sum_{n=0}^{m_{k}}a_{n}z^{n}$
converges uniformly in a neighbourhood of $B(r,0)\cup A$.
 \end{To}

For every domain $D\subset\Bbb{C}$, let $O(\bar{D})$ denote the set of all functions that are
holomorphic in an open neighbourhood of $\bar{D}=D\cup\partial D$.
 
\begin{proof}[Proof Theorem \ref{14} ] Without loss of generality, let $r=1$ and let $c\in\partial\Bbb{E}$. We introduce
the polynomial
\[
q(w)=\frac{c}{2}(w^{p}+w^{p+1}),\quad p\in\Bbb{N}\quad\text{and}\quad p\geq\delta^{-1},
\]
(where the smaller value of $p$ such that $p\geq\delta^{-1}$ gives
the larger domain of overconvergence). Consider the function $f(q(w))=\ds\sum_{n=0}^\infty a_{n}q(w)^{n}$,
which is holomorphic in $q^{-1}(\Bbb{E})=\lbrace w\in\Bbb{C}:|q(w)|<1\rbrace$.
We denote by  $\ds\sum_{n=0}^\infty b_{n}w^{n}$ the Taylor series of $f(q(w))$ about
$0\in q^{-1}(\Bbb{E})$ and by $s_{v}(z)$ and $t_{v}(z)$ the $v$th
partial sums of  $\ds\sum_{n=0}^\infty a_{n}z^{n}$ and $\ds\sum_{n=0}^\infty b_{n}w^{n}$, respectively.
We claim that
\begin{equation}
\quad|t_{(p+1)m_{k}}(w)-s_{m_{k}}(q(w))|\rightarrow0\quad as\quad k\rightarrow\infty\quad where\,\quad w\in E.\label{claim}
\end{equation}
We have $t_{(p+1)m_{k}}(w)=\ds\sum_{n=0}^{(p+1)m_{k}}b_{n}w^{n}$,
and 
\begin{eqnarray*}
s_{m_{k}}(q(w)) & = & \sum_{r=0}^{m_{k}}a_{r}q(w)^{r}=\sum_{r=0}^{m_{k}}a_{r}(\frac{c}{2}(w^{p}+w^{p+1}))^{r}\\
 & = & \sum_{r=0}^{m_{k}}a_{r}(\frac{c}{2})^{r}(C_{r}^{r}w^{pr}+...+C_{0}^{r}w^{(p+1)r})\\
 & = & \sum_{r=0}^{m_{k}}a_{r}(\frac{c}{2})^{r}\sum_{l=0}^{r}C_{l}^{r}\, w^{r(p+1)-l}\,.
\end{eqnarray*}
Each polynomial $(C_{r}^{r}w^{pr}+...+C_{0}^{r}w^{(p+1)r})$ has a
range of powers $pr\leq n\leq(p+1)r$. If $0\leq r\leq m_{k}$, then
$0\leq n\leq(p+1)m_{k}$. Subsequently, $s_{m_{k}}(q(w))$ has a degree
no greater than $(p+1)m_{k}$. Thus, by setting $n=r(p+1)-l$, we
obtain that 
\begin{eqnarray*}
s_{m_{k}}(q(w)) & = & \sum_{n=0}^{(p+1)m_{k}}\sum_{r=[\frac{n}{p+1}]}^{[\frac{n}{p}]}a_{r}\,(\frac{c}{2})^{r}\, C_{r(p+1)-n}^{r}w^{n}\\& = & \sum_{n=0}^{(p+1)m_{k}}d_{n}^{(k)}w^{n},
\end{eqnarray*}
 where $d_{n}^{(k)}=\ds\sum_{r=[\frac{n}{p+1}]}^{[\frac{n}{p}]}a_{r}\,(\frac{c}{2})^{r}\, C_{r(p+1)-n}^{r}\,$.
This means that by Weierstrass' double series theorem \amscite
{burckel2012introduction}*{I.8.4.2},
$\sum b_{n}w^{n}$ and $\sum d_{n}^{(k)}w^{n}$ come from the series
$\sum a_{r}q(w)^{r}$ by multiplying out the polynomials $q(w)^{r}$
and grouping the resulting series $\sum a_{r}(...)$ according to
powers of $w$. Therefore, in order to prove \eqref{claim}, we
need to compute the difference between $t_{(p+1)m_{k}}(w)$ and $s_{m_{k}}(q(w))$,
and then show that it tends to zero uniformly. In fact, we have $t_{(p+1)m_{k}}(w)=\ds s_{m_{k}}(q(w))+(a\, contribution\, of\,\sum_{r>m_{k}}a_{r}q(w)^{r})$,
because when $n$ is from $0$ to $pm_{k}$, the coefficients $b_{n}$
and $d_{n}^{(k)}$ of the partial sums $t_{(p+1)m_{k}}$ and $s_{m_{k}}(q(w))$
 are equal, respectively, i.e. $b_{n}=d_{n}^{(k)}$, for $0\leq n\leq pm_{k}$.
In addition, in the partial sum of $\ds\sum_{r}a_{r}q(w)^{r}$ when
$m_{k+1}>r>n_{k}$, by $(i)$, each polynomial $a_{r}q(w)^{r}$, $r>n_{k}$,
contains monomial $d_{n}^{(k)}w^{n}$ with $n>p n_{k}$. Since $p n_{k}>pm_{k}+p\delta m_{k}\geq(p+1)m_{k}$,
$p\geq\delta^{-1}$, no such monomial contributes to $t_{(p+1)m_{k}}(w)$,
which is a polynomial of degree no greater than $(p+1)m_{k}$. However,
the contribution exists from each $a_{r}q(w)^{r}$ where $|a_{r}|<\frac{1}{m_{k}^{2}}c_{r}$, $c_r$ is the notation of quasi-Ostrowski series,
$m_{k}<r<n_{k}$, i.e., it exists from the monomial $d_{n}^{(k)}w^{n}$
where $pm_{k}<n<p n_{k}$. Since $t_{(p+1)m_{k}}$ has degree no greater
than $(p+1)m_{k}$, and $b_{n}=d_{n}^{(k)}$, for $0\leq n\leq pm_{k}$,
the total contributions to the partial sum $t_{(p+1)m_{k}}$ is computed
in the range of powers $pm_{k}<n\leq(p+1)m_{k}$ as follows.

\begin{eqnarray*}
	|b_{n}-d_{n}^{(k)}|&=&\left|\ds\sum_{r=[\frac{n}{p+1}]}^{[\frac{n}{p}]}a_{r}\,(\frac{c}{2})^{r}\, C_{r(p+1)-n}^{r}\right|\\&\leq &\sum_{r=[\frac{n}{p+1}]}^{[\frac{n}{p}]}|a_{r}|\,\frac{1}{2^{r}}\, C_{r(p+1)-n}^{r}\quad where\quad pm_{k}<n\leq(p+1)m_{k}.
\end{eqnarray*}

Since $C_{r(p+1)-n}^{r}\leq2^{r}$, then for $pm_{k}<n\leq(p+1)m_{k}$ and by using  $c_r$ the notation of quasi-Ostrowski series, we have 
\begin{eqnarray*}
|b_{n}-d_{n}^{(k)}| & \leq & \sum_{r=[\frac{n}{p+1}]}^{[\frac{n}{p}]}|a_{r}|<\frac{1}{m_{k}^{2}}\sum_{r=[\frac{n}{p+1}]}^{[\frac{n}{p}]}c_{r}\\
 & \leq & \frac{1}{m_{k}^{2}}\sum_{r=[\frac{n}{p+1}]}^{[\frac{n}{p}]}\max_{[\frac{n}{p+1}]\leq r\leq[\frac{n}{p}]}c_{r}\\
 & \leq & \frac{1}{m_{k}^{2}}\left(\left[\frac{n}{p}\right]-\left[\frac{n}{p+1}\right]+1\right)c_{[\frac{n}{p+1}]}\\
 & < & \frac{1}{m_{k}^{2}}\left(\frac{n}{p(p+1)}+2\right)c_{[\frac{n}{p+1}]}\,.
\end{eqnarray*}
Therefore, 
\[
|b_{n}-d_{n}^{(k)}|\leq\left\{ \begin{array}{ll}
0 & \mbox{if }\quad0\leq n\leq pm_{k}\\
\frac{1}{m_{k}^{2}}\left(\frac{n}{p(p+1)}+2\right)c_{[\frac{n}{p+1}]} & \mbox{if }\quad pm_{k}<n\leq(p+1)m_{k}\,.
\end{array}\right.
\]
Consequently, 
\begin{eqnarray*}
|t_{(p+1)m_{k}}(w)-s_{m_{k}}(q(w))| & = & \left|\sum_{n=0}^{(p+1)m_{k}}(b_{n}-d_{n}^{(k)})w^{n}\right|=\left|\sum_{n=pm_{k}+1}^{(p+1)m_{k}}(b_{n}-d_{n}^{(k)})w^{n}\right|\\
 & < & \frac{1}{m_{k}^{2}}\sum_{n=pm_{k}+1}^{(p+1)m_{k}}\left(\frac{n}{p(p+1)}+2\right)c_{[\frac{n}{p+1}]}\\
 & \leq & \frac{1}{m_{k}^{2}}\sum_{n=pm_{k}+1}^{(p+1)m_{k}}\left(\frac{(p+1)m_{k}}{p(p+1)}+2\right) c_{[\frac{p m_k}{p+1}]}\\
 & \leq & \frac{1}{m_{k}^{2}}\left(\frac{m_{k}^{2}}{p}+2m_{k}\right)c_{[\frac{pm_{k}}{p+1}]}\rightarrow0\quad as\quad k\rightarrow\infty\,,
\end{eqnarray*}
because by hypothesis given in Definition \mbox{\ref{17}} of a quasi-Ostrowski series,
$c_{[\frac{pm_{k}}{p+1}]}$  decreases to zero as $k\rightarrow\infty$.
Thus, \eqref{claim} was confirmed.

We have $q^{-1}(\Bbb{E})\supset\overline{\Bbb{E}}\backslash\lbrace1\rbrace$,
since $|1+w|<2$. Hence $|q(w)|<1$ for all $w\in\overline{\Bbb{E}}\backslash\lbrace1\rbrace$.
Set $g(w)=f(q(w))-s_{m_{k}}(q(w))$. The function $g(w)\in O(q^{-1}(\Bbb{E}))$
is thus holomorphic at every point of $\overline{\Bbb{E}}\backslash\lbrace1\rbrace$
for each $m_{k}\in\Bbb{N}$. Now, if $c\in A$, then $g$ is also
holomorphic at $1$ since $q(1)=c$. Thus, the sequence of sections
$(g_{(p+1)m_{k}})$ converges to zero in an open disk $B\supset\overline{\Bbb{E}}$.
Then, by \eqref{claim},  the sequence $(s_{m_{k}}(z))$ now converges uniformly
in $q(B)$. Since $q(B)$ is a domain containing $c=q(1)$,  then
$(s_{m_{k}}(z))$ converges  uniformly in a neighbourhood of any point
$c\in A$. \end{proof}

In the following, we provide our proof of Theorem \ref{8}, which is a modification of Hadamard's gap theorem  (Theorem \ref{Hadamard}) given in the Introduction. First, however, we need to present the definition of a quasi-Hadamard lacunary series.

 \begin{de}\label{18} 
 An infinite power series
$\ds\sum_{v=0}^\infty a_{v}z^{v}$ is called a \emph{quasi-Hadamard lacunary} series if there
exists a $\delta>0$, a summable positive decreasing sequence $(c_{r}),\, r=0,1,2,\dots$,
$c_{r}\searrow0$,
and an increasing sequence $(m_{v})$ of non-negative integers for $v=0,1,2,\dots$, such that: 

\begin{enumerate}[label=(\roman*)]\item $m_{v+1}-m_{v}>\delta m_{v}$,\quad$v=0,1,2,\dots$; 
\item $|a_{j}|\leq\frac{1}{m_v^2}c_{j}$,\quad{}\text{if}\quad{}$m_{v}<j<m_{v+1}$; and
\item the series  $\ds\sum_{v=0}^\infty |a_{m_{v}}|$ is divergent. 
\end{enumerate}
\end{de} 
Now, we present our new theorem as follows.
\begin{To}[On a quasi-Hadamard lacunary series]\label{8} Every quasi-Hadamard lacunary series
$\ds\sum_{n=0}^\infty a_{n}z^{n}$ with radius of convergence $r>0$ has the disc
$B(r,0)$ as a domain of holomorphy. 
\end{To}
\begin{rem}\begin{enumerate}
\item In Definition \ref{18}, condition $(iii)$ that the series $\ds\sum_{v=0}^\infty |a_{m_{v}}|$ is divergent  is necessary, otherwise the previous theorem would not hold. For example, let $\ds\sum_{n=1}^\infty a_{n}z^{n}=\sum_{n=1}^\infty\frac{1}{n^2}z^n$. Notice that conditions (i) and (ii) are satisfied, but we have $\ds\sum_{v=1}^\infty|a_{m_{v}}|<\infty$. On the other hand,  denote by $s_{m_{v+1}}$ the $m_{v+1}$ partial sums of $\ds\sum_{n=1}^\infty a_{n}z^{n}$. Then,  
\begin{equation*}
	|{s_{m_{v+1}}|=\left|\sum_{n=1}^{m_{v+1}} a_{n}z^{n}\right|}\leq \sum_{n=1}^{m_{v+1}}|\frac{1}{n^2}|.
\end{equation*} 
Since, $\ds\lim_{v\rightarrow\infty}|s_{m_{v+1}}|=L<\infty$,
  the conclusion of Theorem \ref{8} does not hold.
\item The series of quasi-Hadamard lacunary series is a subclass of quasi-Ostrowski series.
\end{enumerate}
\end{rem} 
 \begin{proof}[Proof Theorem \ref{8}] Let
 $s_{n}(z)$ $=\ds\sum_{j=0}^{n}a_{j}z^{j}$ be the $n$th partial sums of $f(z)=\ds\sum_{n=0}^\infty a_{n}z^{n}$.
Consider the partial sums $s_{m_k}(z)$, where the sequence of $(m_{k})$ is defined
as in Definition \ref{18} of a quasi-Hadamard lacunary series.
For the partial sums of $s_{n}(z)$ we have that 
\[
s_{n}(z)=\left\{ \begin{array}{ll}
s_{m_{k}}(z) & \mbox{if }\quad n=m_{k},\\
s_{m_{k}}(z)+\ds\sum_{j=m_{k}+1}^{n}a_{j}z^{j} & \mbox{if }\quad m_{k}<n<m_{k+1}\,,
\end{array}\right.
\]
where $|a_{j}|\leq\frac{1}{m_k^2}c_{j}$, $c_j$ is the notation used in the quasi-Hadamard lacunary series, and  $m_{k}<j<m_{k+1}$. Therefore,

\begin{equation}
\begin{split}\label{20}|s_{n}(z)-s_{m_{k}}(z)| & =\left|\sum_{j=0}^{n}a_{j}z^{j}-\sum_{j=0}^{m_{k}}a_{j}z^{j}\right|\\
 & =\left|\sum_{j=m_{k}+1}^{n}a_{j}z^{j}\right|\\&\leq\sum_{j=m_{k}+1}^{m_{k+1}-1}|a_{j}|\rightarrow0\quad as\quad k\rightarrow\infty,
\end{split}
\end{equation}
since for $m_k<j<m_{k+1}$, we have  $\ds\sum_j|a_{j}|\leq \frac{1}{m_k^2} \sum_j c_{j}<\infty$. Thus, in \eqref{20} $\ds\sum_{j=m_{k}+1}^{m_{k+1}-1}|a_{j}|\rightarrow0$ when $k\rightarrow\infty$.

Then, by \eqref{20}, the sequence of partial sums $(s_{n}(z))$ converges in
the same domain as $(s_{m_{k}}(z))$. Thus, the sequence $(s_{m_{k}}(z))$ 
diverges at every point  $\zeta\notin{\overline{B(r,0)}}$. Hence,
by Theorem  \ref{14} of overconvergence, all the points of  $\partial B(r,0)$
are singular points of $f$.
 \end{proof} 
 
\maketitle
\section*{Acknowledgment}
 I am highly indebted to my supervisor (Vladimir Kisil) for his guidance, and his constant supervision of this work. My sincere thanks also go to my government of Saudi Arabia who fund me for this study.

\bibliography{Analysis1}

\providecommand{\arXiv}[1]{\href{http://arXiv.org/abs/#1}{arXiv:#1}}
\begin{thebibliography}{1}

\bibitem{BreuerSimon11a}
Jonathan Breuer and Barry Simon.
\newblock Natural boundaries and spectral theory.
\newblock {\em Adv. Math.}, 226(6):4902--4920, 2011.

\bibitem{burckel2012introduction}
Robert~B Burckel.
\newblock {\em An introduction to classical complex analysis}, volume~1.
\newblock Birkh{\"a}user, 2012.

\bibitem{Kisil97c}
Vladimir~V. Kisil.
\newblock Analysis in {$\mathbf{R}\sp {1,1}$} or the principal function theory.
\newblock {\em Complex Variables Theory Appl.}, 40(2):93--118, 1999.
\newblock \arXiv{funct-an/9712003}.

\bibitem{Kisil11c}
Vladimir~V. Kisil.
\newblock {E}rlangen programme at large: an {O}verview.
\newblock In S.V. Rogosin and A.A. Koroleva, editors, {\em Advances in Applied
  Analysis}, chapter~1, pages 1--94. Birkh\"auser Verlag, Basel, 2012.
\newblock \arXiv{1106.1686}.

\bibitem{Kisil12d}
Vladimir~V. Kisil.
\newblock The real and complex techniques in harmonic analysis from the point
  of view of covariant transform.
\newblock {\em Eurasian Math. J.}, 5:95--121, 2014.
\newblock \arXiv{1209.5072}.
  \href{http://emj.enu.kz/images/pdf/2014/5-1-4.pdf}{On-line}.

\bibitem{luhpower}
Wolfgang Luh.
\newblock Power series---the structure of {H}.-{O}.-gaps.
\newblock {\em Ann. Univ. Sci. Budapest. Sect. Comput.}, 39:303--309, 2013.

\bibitem{MR580154}
George P\'olya and G\'abor Szeg\H~o.
\newblock {\em Problems and theorems in analysis. {I}}, volume 193 of {\em
  Grundlehren der Mathematischen Wissenschaften [Fundamental Principles of
  Mathematical Sciences]}.
\newblock Springer-Verlag, Berlin-New York, 1978.
\newblock Series, integral calculus, theory of functions, Translated from the
  German by D. Aeppli, Corrected printing of the revised translation of the
  fourth German edition.

\bibitem{remmert2013classical}
Reinhold Remmert.
\newblock {\em Classical topics in complex function theory}, volume 172.
\newblock Springer Science \& Business Media, 2013.

\end{thebibliography}
 \bibliographystyle{plain} 
\end{document}